\documentclass{amsart}
\usepackage{amssymb}
\usepackage{amsfonts}

\newtheorem{theorem}{Theorem}
\theoremstyle{plain}

\newtheorem{corollary}{Corollary}

\newtheorem{definition}{Definition}

\newtheorem{lemma}{Lemma}

\numberwithin{equation}{section}
\input{tcilatex}

\begin{document}
\title[Inequalities for Godunova-Levin Class Functions]{Integral
Inequalities for functions whose 3rd derivatives belong to $Q(I)$}
\author{M.Emin Ozdemir$^{\blacklozenge }$}
\address{$^{\blacklozenge }$ATATURK UNIVERSITY, K.K. EDUCATION FACULTY,
DEPARTMENT OF MATHEMATICS, 25240 CAMPUS, ERZURUM, TURKEY}
\email{emos@atauni.edu.tr}
\author{Merve Avci Ardic$^{\ast ,\diamondsuit }$}
\thanks{$^{\diamondsuit }$Corresponding Author}
\address{$^{\ast }$ADIYAMAN UNIVERSITY, FACULTY OF SCIENCE AND ART,
DEPARTMENT OF MATHEMATICS, 02040, ADIYAMAN, TURKEY}
\email{mavci@posta.adiyama.edu.tr}
\author{Mustafa Gurbuz$^{\blacktriangle }$}
\address{$^{\blacktriangle }$A\u{g}r\i\ \.{I}brahim \c{C}e\c{c}en
University, Faculty of Education, Department of Mathematics, 04100, A\u{g}%
r\i , Turkey}
\email{mgurbuz@agri.edu.tr}

\begin{abstract}
In this paper, we obtain some new inequalities of Hermite-Hadamard type and
Simpson type for functions whose third derivatives belong to Godunova-Levin
class.
\end{abstract}

\maketitle

\section{introduction}

Following inequalities are well known in the literature as Hermite-Hadamard
inequality and Simpson inequality respectively:

\begin{theorem}
\label{teo 1.1} Let $f:I\subseteq 
\mathbb{R}
\rightarrow 
\mathbb{R}
$be a convex function on the interval $I$ of real numbers and $a,b\in I$
with $a<b.$Then, the following double inequality holds 
\begin{equation*}
f\left( \frac{a+b}{2}\right) \leq \frac{1}{b-a}\int_{a}^{b}f(x)dx\leq \frac{%
f(a)+f(b)}{2}.
\end{equation*}
\end{theorem}

\begin{theorem}
\label{teo 1.2} Let $f:\left[ a,b\right] \rightarrow 
\mathbb{R}
$ be a four times continuously differentiable mapping on $\left( a,b\right) $
and $\left\Vert f^{(4)}\right\Vert _{\infty }=\sup\limits_{x\in \left(
a,b\right) }\left\vert f^{(4)}\left( x\right) \right\vert <\infty .$ Then,
the following inequality holds:%
\begin{equation*}
\left\vert \frac{1}{3}\left[ \frac{f(a)+f\left( b\right) }{2}+2f\left( \frac{%
a+b}{2}\right) \right] -\frac{1}{b-a}\dint\limits_{a}^{b}f(x)dx\right\vert
\leq \frac{1}{2880}\left\Vert f^{(4)}\right\Vert _{\infty }\left( b-a\right)
^{4}.
\end{equation*}
\end{theorem}

In 1985, E. K. Godunova and V. I. Levin introduced the following class of
functions (see \cite{GL}):

\begin{definition}
\label{def 1.1} A map $f:I\rightarrow 
\mathbb{R}
$ is said to belong to the class $Q(I)$ if it is nonnegative and for all $%
x,y\in I$ and $\lambda \in \left( 0,1\right) ,$ satisfies the inequality%
\begin{equation*}
f(\lambda x+(1-\lambda )y)\leq \frac{f(x)}{\lambda }+\frac{f(y)}{1-\lambda }.
\end{equation*}
\end{definition}

In \cite{MK}, Moslehian and Kian obtained Hermite-Hadamard and Ostrowski
type inequalities for functions whose first derivatives belong to $Q(I).$

To obtain our new results, it is necessary two lemmas .

\begin{lemma}
\label{lem 1.1} \cite{CQ} Let $f:I\subseteq 
\mathbb{R}
\rightarrow 
\mathbb{R}
$ be a three times differentiable function on $I^{\circ }$ with $a,b\in I$, $%
a<b.$ If $f^{\prime \prime \prime }\in L[a,b],$ then%
\begin{eqnarray*}
&&\frac{f(a)+f(b)}{2}-\frac{1}{b-a}\int_{a}^{b}f(x)dx-\frac{b-a}{12}\left[
f^{\prime }(b)-f^{\prime }(a)\right] \\
&=&\frac{\left( b-a\right) ^{3}}{12}\int_{0}^{1}t\left( 1-t\right) \left(
2t-1\right) f^{\prime \prime \prime }\left( ta+(1-t)b\right) dt.
\end{eqnarray*}
\end{lemma}

\begin{lemma}
\label{lem 1.2} \cite{AH} Let $f^{\prime \prime }:I\subseteq 
\mathbb{R}
\rightarrow 
\mathbb{R}
$ be an absolutely continuous function on $I^{\circ }$ such that $f^{\prime
\prime \prime }\in L[a,b],$ where $a,b\in I$, $a<b.$ Then 
\begin{eqnarray*}
&&\int_{a}^{b}f(x)dx-\frac{b-a}{6}\left[ f(a)+4f\left( \frac{a+b}{2}\right)
+f(b)\right] \\
&=&\left( b-a\right) ^{4}\int_{0}^{1}p(t)f^{\prime \prime \prime }\left(
ta+(1-t)b\right) dt,
\end{eqnarray*}%
where%
\begin{equation*}
p(t)=\left\{ 
\begin{array}{ccc}
\frac{1}{6}t^{2}\left( t-\frac{1}{2}\right) , &  & \text{if }t\in \left[ 0,%
\frac{1}{2}\right] \\ 
&  &  \\ 
\frac{1}{6}\left( t-1\right) ^{2}\left( t-\frac{1}{2}\right) , &  & \text{if 
}t\in \left( \frac{1}{2},1\right] .%
\end{array}%
\right.
\end{equation*}
\end{lemma}

In this paper, using Lemma \ref{lem 1.1} and Lemma \ref{lem 1.2}, we obtain
some new inequalities for functions whose third derivatives belong to $Q(I).$

\section{Main Results}

We obtain the following new inequalities via Lemma \ref{lem 1.1}.

\begin{theorem}
\label{teo 2.1} Let $f^{\prime \prime }:I\subseteq 
\mathbb{R}
\rightarrow 
\mathbb{R}
$ be an absolutely continuous function on $I^{\circ }$ such that $f^{\prime
\prime \prime }\in L[a,b],$ where $a,b\in I$, $a<b.$ If $\left\vert
f^{\prime \prime \prime }\right\vert ^{q}$ belongs to $Q(I),$ then the
following inequality holds%
\begin{eqnarray*}
&&\left\vert \frac{f(a)+f(b)}{2}-\frac{1}{b-a}\int_{a}^{b}f(x)dx-\frac{b-a}{%
12}\left[ f^{\prime }(b)-f^{\prime }(a)\right] \right\vert \\
&\leq &\frac{\left( b-a\right) ^{3}}{12}\left( \frac{1}{16}\right) ^{1-\frac{%
1}{q}}\left[ \frac{\left\vert f^{\prime \prime \prime }(a)\right\vert
^{q}+\left\vert f^{\prime \prime \prime }(b)\right\vert ^{q}}{4}\right] ^{%
\frac{1}{q}}
\end{eqnarray*}%
for $q\geq 1.$
\end{theorem}

\begin{proof}
From Lemma \ref{lem 1.1} and using the power mean inequality, we have%
\begin{eqnarray*}
&&\left\vert \frac{f(a)+f(b)}{2}-\frac{1}{b-a}\int_{a}^{b}f(x)dx-\frac{b-a}{%
12}\left[ f^{\prime }(b)-f^{\prime }(a)\right] \right\vert \\
&\leq &\frac{\left( b-a\right) ^{3}}{12}\left( \int_{0}^{1}t(1-t)\left\vert
2t-1\right\vert dt\right) ^{1-\frac{1}{q}}\left(
\int_{0}^{1}t(1-t)\left\vert 2t-1\right\vert \left\vert f^{\prime \prime
\prime }\left( ta+(1-t\right) b)\right\vert ^{q}dt\right) ^{\frac{1}{q}}.
\end{eqnarray*}%
Then, since $\left\vert f^{\prime \prime \prime }\right\vert ^{q}$ belongs
to $Q(I)$, we can write for $t\in (0,1)$%
\begin{equation*}
\left\vert f^{\prime \prime \prime }\left( ta+(1-t\right) b)\right\vert
^{q}\leq \frac{\left\vert f^{\prime \prime \prime }(a)\right\vert ^{q}}{t}+%
\frac{\left\vert f^{\prime \prime \prime }(b)\right\vert ^{q}}{1-t}.
\end{equation*}%
Hence,%
\begin{eqnarray*}
&&\left\vert \frac{f(a)+f(b)}{2}-\frac{1}{b-a}\int_{a}^{b}f(x)dx-\frac{b-a}{%
12}\left[ f^{\prime }(b)-f^{\prime }(a)\right] \right\vert \\
&\leq &\frac{\left( b-a\right) ^{3}}{12}\left( \frac{1}{16}\right) ^{1-\frac{%
1}{q}}\left( \int_{0}^{1}t(1-t)\left\vert 2t-1\right\vert \left[ \frac{%
\left\vert f^{\prime \prime \prime }(a)\right\vert ^{q}}{t}+\frac{\left\vert
f^{\prime \prime \prime }(b)\right\vert ^{q}}{1-t}\right] dt\right) ^{\frac{1%
}{q}} \\
&=&\frac{\left( b-a\right) ^{3}}{12}\left( \frac{1}{16}\right) ^{1-\frac{1}{q%
}}\left( \int_{0}^{1}\left[ (1-t)\left\vert 2t-1\right\vert \left\vert
f^{\prime \prime \prime }(a)\right\vert ^{q}+t\left\vert 2t-1\right\vert
\left\vert f^{\prime \prime \prime }(b)\right\vert ^{q}\right] dt\right) ^{%
\frac{1}{q}} \\
&=&\frac{\left( b-a\right) ^{3}}{12}\left( \frac{1}{16}\right) ^{1-\frac{1}{q%
}}\left[ \frac{\left\vert f^{\prime \prime \prime }(a)\right\vert
^{q}+\left\vert f^{\prime \prime \prime }(b)\right\vert ^{q}}{4}\right] ^{%
\frac{1}{q}},
\end{eqnarray*}%
where%
\begin{equation*}
\int_{0}^{1}t(1-t)\left\vert 2t-1\right\vert dt=\frac{1}{16}
\end{equation*}%
and%
\begin{equation*}
\int_{0}^{1}t\left\vert 2t-1\right\vert dt=\int_{0}^{1}(1-t)\left\vert
2t-1\right\vert dt=\frac{1}{4}.
\end{equation*}%
The proof is completed.
\end{proof}

\begin{corollary}
\label{co 1.1} In Theorem \ref{teo 2.1}, if we choose $q=1$ we obtain the
following inequality%
\begin{eqnarray*}
&&\left\vert \frac{f(a)+f(b)}{2}-\frac{1}{b-a}\int_{a}^{b}f(x)dx-\frac{b-a}{%
12}\left[ f^{\prime }(b)-f^{\prime }(a)\right] \right\vert \\
&\leq &\frac{\left( b-a\right) ^{3}}{48}\left[ \left\vert f^{\prime \prime
\prime }(a)\right\vert +\left\vert f^{\prime \prime \prime }(b)\right\vert %
\right] .
\end{eqnarray*}
\end{corollary}

\begin{theorem}
\label{teo 2.2} Let $f^{\prime \prime }:I\subseteq 
\mathbb{R}
\rightarrow 
\mathbb{R}
$ be an absolutely continuous function on $I^{\circ }$ such that $f^{\prime
\prime \prime }\in L[a,b],$ where $a,b\in I$, $a<b.$ If $\left\vert
f^{\prime \prime \prime }\right\vert ^{q}$ belongs to $Q(I)$ and $q>1,$ then
the following inequality holds%
\begin{eqnarray*}
&&\left\vert \frac{f(a)+f(b)}{2}-\frac{1}{b-a}\int_{a}^{b}f(x)dx-\frac{b-a}{%
12}\left[ f^{\prime }(b)-f^{\prime }(a)\right] \right\vert \\
&\leq &\frac{\left( b-a\right) ^{3}}{12}\frac{\left( \beta \left(
q,q+1\right) \right) ^{\frac{1}{q}}}{\left( p+1\right) ^{\frac{1}{p}}}\left[
\left\vert f^{\prime \prime \prime }(a)\right\vert +\left\vert f^{\prime
\prime \prime }(b)\right\vert \right] ^{\frac{1}{q}}
\end{eqnarray*}%
where $\frac{1}{p}+\frac{1}{q}=1,$ and $\beta \left( ,\right) $ is Euler
Beta function.
\end{theorem}

\begin{proof}
Since $\left\vert f^{\prime \prime \prime }\right\vert ^{q}$ belongs to $%
Q(I),$ from Lemma \ref{lem 1.1} and using the H\"{o}lder inequality we have%
\begin{eqnarray*}
&&\left\vert \frac{f(a)+f(b)}{2}-\frac{1}{b-a}\int_{a}^{b}f(x)dx-\frac{b-a}{%
12}\left[ f^{\prime }(b)-f^{\prime }(a)\right] \right\vert \\
&\leq &\frac{\left( b-a\right) ^{3}}{12}\left( \int_{0}^{1}\left\vert
2t-1\right\vert ^{p}dt\right) ^{\frac{1}{p}}\left(
\int_{0}^{1}t^{q}(1-t)^{q}\left\vert f^{\prime \prime \prime }\left(
ta+(1-t\right) b)\right\vert ^{q}dt\right) ^{\frac{1}{q}} \\
&\leq &\frac{\left( b-a\right) ^{3}}{12}\left( \frac{1}{p+1}\right) ^{\frac{1%
}{p}}\left( \int_{0}^{1}t^{q}(1-t)^{q}\left[ \frac{\left\vert f^{\prime
\prime \prime }(a)\right\vert ^{q}}{t}+\frac{\left\vert f^{\prime \prime
\prime }(b)\right\vert ^{q}}{1-t}\right] dt\right) ^{\frac{1}{q}} \\
&=&\frac{\left( b-a\right) ^{3}}{12}\frac{\left( \beta \left( q,q+1\right)
\right) ^{\frac{1}{q}}}{\left( p+1\right) ^{\frac{1}{p}}}\left[ \left\vert
f^{\prime \prime \prime }(a)\right\vert ^{q}+\left\vert f^{\prime \prime
\prime }(b)\right\vert ^{q}\right] ^{\frac{1}{q}},
\end{eqnarray*}%
which completes the proof.
\end{proof}

\begin{theorem}
\label{teo 2.3} Under the assumptions of Theorem \ref{teo 2.2}, we have the
following inequality%
\begin{eqnarray*}
&&\left\vert \frac{f(a)+f(b)}{2}-\frac{1}{b-a}\int_{a}^{b}f(x)dx-\frac{b-a}{%
12}\left[ f^{\prime }(b)-f^{\prime }(a)\right] \right\vert \\
&\leq &\frac{\left( b-a\right) ^{3}}{24}\left( \frac{1}{\left( p+1\right)
\left( p+3\right) }\right) ^{\frac{1}{p}}\left[ \left\vert f^{\prime \prime
\prime }(a)\right\vert ^{q}+\left\vert f^{\prime \prime \prime
}(b)\right\vert ^{q}\right] ^{\frac{1}{q}}
\end{eqnarray*}%
where $\frac{1}{p}+\frac{1}{q}=1.$
\end{theorem}

\begin{proof}
Since $\left\vert f^{\prime \prime \prime }\right\vert ^{q}$ belongs to $%
Q(I),$ from Lemma \ref{lem 1.1} and using the H\"{o}lder inequality we have%
\begin{eqnarray*}
&&\left\vert \frac{f(a)+f(b)}{2}-\frac{1}{b-a}\int_{a}^{b}f(x)dx-\frac{b-a}{%
12}\left[ f^{\prime }(b)-f^{\prime }(a)\right] \right\vert \\
&\leq &\frac{\left( b-a\right) ^{3}}{12}\left( \int_{0}^{1}t\left(
1-t\right) \left\vert 2t-1\right\vert ^{p}dt\right) ^{\frac{1}{p}}\left(
\int_{0}^{1}t(1-t)\left\vert f^{\prime \prime \prime }\left( ta+(1-t\right)
b)\right\vert ^{q}dt\right) ^{\frac{1}{q}} \\
&\leq &\frac{\left( b-a\right) ^{3}}{12}\left( \int_{0}^{1}t\left(
1-t\right) \left\vert 2t-1\right\vert ^{p}dt\right) ^{\frac{1}{p}}\left(
\int_{0}^{1}t(1-t)\left[ \frac{\left\vert f^{\prime \prime \prime
}(a)\right\vert ^{q}}{t}+\frac{\left\vert f^{\prime \prime \prime
}(b)\right\vert ^{q}}{1-t}\right] dt\right) ^{\frac{1}{q}} \\
&=&\frac{\left( b-a\right) ^{3}}{12}\left( \frac{1}{2\left( p+1\right)
\left( p+3\right) }\right) ^{\frac{1}{p}}\left[ \frac{\left\vert f^{\prime
\prime \prime }(a)\right\vert ^{q}+\left\vert f^{\prime \prime \prime
}(b)\right\vert ^{q}}{2}\right] ^{\frac{1}{q}},
\end{eqnarray*}%
where we used 
\begin{equation*}
\int_{0}^{1}t\left( 1-t\right) \left\vert 2t-1\right\vert ^{p}dt=\frac{1}{%
2\left( p+1\right) \left( p+3\right) }.
\end{equation*}%
The proof is completed.
\end{proof}

Following result is obtained via Lemma \ref{lem 1.2}.

\begin{theorem}
\label{teo 2.4} Let $f^{\prime \prime }:I\subseteq 
\mathbb{R}
\rightarrow 
\mathbb{R}
$ be an absolutely continuous function on $I^{\circ }$ such that $f^{\prime
\prime \prime }\in L[a,b],$ where $a,b\in I$, $a<b.$ If $\left\vert
f^{\prime \prime \prime }\right\vert ^{q}$ belongs to $Q(I),$ then the
following inequality holds%
\begin{eqnarray*}
&&\left\vert \int_{a}^{b}f(x)dx-\frac{b-a}{6}\left[ f(a)+4f\left( \frac{a+b}{%
2}\right) +f(b)\right] \right\vert \\
&\leq &\frac{\left( b-a\right) ^{4}}{6}\left( \frac{1}{192}\right) ^{1-\frac{%
1}{q}}\left\{ \left( \frac{\left\vert f^{\prime \prime \prime
}(a)\right\vert ^{q}}{48}+\left( \frac{17}{48}-\frac{1}{2}\ln 2\right)
\left\vert f^{\prime \prime \prime }(b)\right\vert ^{q}\right) ^{\frac{1}{q}%
}\right. \\
&&\left. +\left( \left( \frac{17}{48}-\frac{1}{2}\ln 2\right) \left\vert
f^{\prime \prime \prime }(a)\right\vert ^{q}+\frac{\left\vert f^{\prime
\prime \prime }(b)\right\vert ^{q}}{48}\right) ^{\frac{1}{q}}\right\}
\end{eqnarray*}%
for $q\geq 1.$
\end{theorem}

\begin{proof}
Since $\left\vert f^{\prime \prime \prime }\right\vert ^{q}$ belongs to $%
Q(I),$ from Lemma \ref{lem 1.2} and using the power mean inequality we have%
\begin{eqnarray}
&&  \label{1} \\
&&\left\vert \int_{a}^{b}f(x)dx-\frac{b-a}{6}\left[ f(a)+4f\left( \frac{a+b}{%
2}\right) +f(b)\right] \right\vert  \notag \\
&\leq &\frac{\left( b-a\right) ^{4}}{6}\left\{ \left( \int_{0}^{\frac{1}{2}%
}t^{2}\left( \frac{1}{2}-t\right) dt\right) ^{1-\frac{1}{q}}\left( \int_{0}^{%
\frac{1}{2}}t^{2}\left( \frac{1}{2}-t\right) \left\vert f^{\prime \prime
\prime }\left( ta+(1-t\right) b)\right\vert ^{q}dt\right) ^{\frac{1}{q}%
}\right.  \notag \\
&&\left. \left( \int_{\frac{1}{2}}^{1}\left( t-1\right) ^{2}\left( t-\frac{1%
}{2}\right) dt\right) ^{1-\frac{1}{q}}\left( \int_{\frac{1}{2}}^{1}\left(
t-1\right) ^{2}\left( t-\frac{1}{2}\right) \left\vert f^{\prime \prime
\prime }\left( ta+(1-t\right) b)\right\vert ^{q}dt\right) ^{\frac{1}{q}%
}\right\}  \notag \\
&\leq &\frac{\left( b-a\right) ^{4}}{6}\left( \frac{1}{192}\right) ^{1-\frac{%
1}{q}}\left\{ \left( \int_{0}^{\frac{1}{2}}t^{2}\left( \frac{1}{2}-t\right) %
\left[ \frac{\left\vert f^{\prime \prime \prime }(a)\right\vert ^{q}}{t}+%
\frac{\left\vert f^{\prime \prime \prime }(b)\right\vert ^{q}}{1-t}\right]
dt\right) ^{\frac{1}{q}}\right.  \notag \\
&&\left. \left( \int_{\frac{1}{2}}^{1}\left( t-1\right) ^{2}\left( t-\frac{1%
}{2}\right) \left[ \frac{\left\vert f^{\prime \prime \prime }(a)\right\vert
^{q}}{t}+\frac{\left\vert f^{\prime \prime \prime }(b)\right\vert ^{q}}{1-t}%
\right] dt\right) ^{\frac{1}{q}}\right\} .  \notag
\end{eqnarray}

If we use the inequalities belowe in (\ref{1}), we get the desired result:%
\begin{equation*}
\int_{0}^{\frac{1}{2}}t^{2}\left( \frac{1}{2}-t\right) dt=\int_{\frac{1}{2}%
}^{1}\left( t-1\right) ^{2}\left( t-\frac{1}{2}\right) dt=\frac{1}{192}
\end{equation*}%
and%
\begin{eqnarray*}
\int_{0}^{\frac{1}{2}}t^{2}\left( \frac{1}{2}-t\right) \frac{1}{1-t}dt
&=&\int_{\frac{1}{2}}^{1}\frac{1}{t}\left( t-1\right) ^{2}\left( t-\frac{1}{2%
}\right) dt \\
&=&\frac{17}{48}-\frac{1}{2}\ln 2.
\end{eqnarray*}
\end{proof}

\begin{corollary}
\label{co 1.2} In Theorem \ref{teo 2.4}, if we choose $q=1$ we obtain the
following inequality%
\begin{eqnarray*}
&&\left\vert \int_{a}^{b}f(x)dx-\frac{b-a}{6}\left[ f(a)+4f\left( \frac{a+b}{%
2}\right) +f(b)\right] \right\vert \\
&\leq &\frac{\left( b-a\right) ^{4}}{6}\left( \frac{3}{8}-\frac{1}{2}\ln
2\right) \left[ \left\vert f^{\prime \prime \prime }(a)\right\vert
+\left\vert f^{\prime \prime \prime }(b)\right\vert \right] .
\end{eqnarray*}
\end{corollary}

\end{document}